\documentclass{article}
\usepackage[utf8]{inputenc}
\usepackage{csquotes}
\usepackage[
backend=biber,
style=alphabetic,
sorting=ynt
]{biblatex}
\bibliography{bibliography}
\usepackage{titling}
\usepackage{amsfonts}
\usepackage{amsmath}
\usepackage{amssymb}
\usepackage{amsthm}
\usepackage[english]{babel}
\usepackage[affil-it]{authblk}
\usepackage[onehalfspacing]{setspace}
\usepackage[margin=1in]{geometry}
\usepackage{hyperref}

\newtheorem{theorem}{Theorem}

\newtheorem*{corollary*}{Corollary}

\newtheorem{lemma}{Lemma}
\newtheorem*{theorem*}{Theorem}

\theoremstyle{definition}

\theoremstyle{definition}
\newtheorem{definition}{Definition}

\theoremstyle{theorem}
\newtheorem{proposition}{Proposition}

\theoremstyle{remark}
\newtheorem{remark}{Remark}

\predate{}
\postdate{}
\title{H\"older regularity of the $\bar\partial-$equation on the polydisc}
\author[1]{Yu Jun Loo}
\author[2]{Alexander Tumanov}
\affil[1]{\small Department of Mathematics, University of Michigan, Ann Arbor, MI 48109, USA}
\affil[2]{\small Department of Mathematics, University of Illinois, Urbana, IL 61801, USA}

\date{}

\begin{document}

\maketitle

\begin{abstract}
In this note, we show the existence of a solution operator to the $\bar\partial-$equation in the polydisc that preserves H\"older regularity. This solution operator is constructed using Henkin's formula. It is a well-known fact that solution operators to the $\bar\partial-$equation on product domains do not improve H\"older regularity. Hence, this solution operator is optimal in that regard.
\end{abstract}

\section{Introduction}
It is a classical problem in complex analysis to describe solutions to the $\bar\partial-$equation with estimates in prescribed normed function spaces. The most general result on problems of this type was given by Sergeev and Henkin in \cite{SergeevHenkin}, giving uniform estimates for the $\bar\partial-$equation in any pseudoconvex polyhedron. Recently, the H\"older spaces $C^{k+\alpha}$ on product domains in $\mathbb{C}^n$ have been given some attention, and some results have been published on this matter. In \cite{Pan1}, \cite{Pan2}, a solution operator which loses arbitrarily small amounts of H\"older regularity was found, while in \cite{Zhang} a solution operator which preserves H\"older regularity was found in the case $n =2$. In the papers mentioned, the solution operators were based on Nijenhuis and Woolf's formula in \cite{NijenhuisWoolf}. Direct estimates of the norm of the solution operator using this formula suggests that it loses (arbitrarily small) amounts of H\"older regularity (see \cite{Pan2} \cite{Tumanov}). In this note however, we show the existence of an optimal solution operator that in fact \textit{preserves} H\"older regularity on the polydisc $\mathbb{D}^n \subset \mathbb{C}^n$ for any $n \geq 1$. Indeed, the main theorem of the paper is as follows.
\begin{theorem}
\label{main theorem}
For any integer $k \geq 0$, and $ 0<\alpha < 1$, Let $Z^{k+\alpha}_{(0,1)}(\mathbb{D}^n) \subseteq C^{k+\alpha}_{(0,1)}(\mathbb{D}^n)$ denote the subspace of $\bar\partial-$closed, H\"older $k+ \alpha$, $(0,1)-$forms on the polydisc. Then for all $g \in Z^{k+\alpha}_{(0,1)}(\mathbb{D}^n)$, the equation $\bar\partial u = g$ admits a bounded linear solution operator $$
H:Z^{k+\alpha}_{(0,1)}(\mathbb{D}^n) \rightarrow C^{k+\alpha}(\mathbb{D}^n).$$

Moreover, $H$ is canonical in the sense that for any $f \in C^{k+\alpha}_{(0,1)}(\mathbb{D}^n)$,  $$KH[f] = 0.$$ Here, $$K[u](z) = \frac{1}{(2\pi i)^n}\int_{(\partial\mathbb{D})^n} \frac{u(\zeta)}{\prod^n_{j=1}(\zeta_j - z_j)}d\zeta_1\wedge...\wedge d\zeta_n$$
is the Cauchy torus integral.
\end{theorem}
\begin{remark}
    Note that this definition of canonical is analogous to H\"ormander's canonical solution to the $\bar\partial$-equation over the domain $\mathbb{D}^n$ with apriori $L^2$ estimates. Indeed, recall that H\"ormander's solution is canonical in the sense that it is orthogonal to the kernel of $\bar\partial$ (ie. analytic functions) with respect to a weighted inner product over $\mathbb{D}^n$. Analogously, the property that $KH = 0$ is equivalent to the fact that $B[H[f],g] := \int_{(\partial\mathbb{D})^n} H[f](\zeta)g(\zeta)\ d\zeta_1\wedge...\wedge d\zeta_n = 0$ for all $g \in A(\mathbb{D}^n)\cap C^{k + \alpha}(\mathbb{D}^n)$. In this sense, $H[f]$ is orthogonal to the kernel of $\bar\partial$. Note however, that $B$ as defined above is not a complex inner product, but a symmetric bilinear form.
\end{remark}
\section{Preliminary results}
The proof of the main theorem rests on an analysis of Henkin's weighted formula for solutions to the $\bar\partial-$equation on the polydisc, which was announced in his survey paper \cite{Henkin} of 1985. The simplest case of this formula, obtained by setting all weights equal to $0$, has the following form.
\newline

\begin{theorem}
\label{Henkins formula theorem}
Let $Z_{(0,1)}(\bar{\mathbb{D}}^n) \subseteq  C_{(0,1)}(\bar{\mathbb{D}}^n)$ denote the space of (uniformly) continuous, $\bar\partial-$closed $(0,1)-$forms on the polydisc. Fix $g \in  Z_{(0,1)}(\bar{\mathbb{D}}^n)$. Then, $$u(z) = H[g](z) = - \sum^{n-1}_{r=0}\sum_{|J| = r} (-1)^{c(n,r)}\int_{\gamma_J(z)}g(\zeta)\wedge H_{J}(\zeta,z)$$

is a distributional solution to the equation $\bar\partial u = g $.
Here, $c(n,r)$ is an integer depending only on the constants $n,r$, while the sum ranges over all ordered $r$-tuples $J = (j_1,...,j_r)$ such that $\{j_1,...,j_r\}$ is a size $r$ subset of $\{1,...,n\}$. The complement of $J$ in $\{1,...,n\}$ is denoted by $\{k_1,...,k_{n-r}\}$, while the region of integration $\gamma_J(z)$ is given by $$\gamma_J(z) = \{\zeta \in \mathbb{D}^n : \zeta_{j_1} = z_{j_1},...,  \zeta_{j_r} = z_{j_r}, |z_{j_1}| \geq ... \geq |z_{j_r}| \geq |\zeta_{k_1}| = ... = |\zeta_{k_{n-r}}| \}.$$ The kernel of integration is the $n-r$ form $$H_{J}(\zeta,z) = \frac{1}{(2\pi i)^{n-r}}\cdot\bigwedge^{n-r}_{s = 1}\frac{d\zeta_{k_s}}{\zeta_{k_s} - z_{k_s}}.$$

Moreover, $KH \equiv0$ where $K$ is the Cauchy torus integral.
\end{theorem}
\medskip

A version of this formula with weights equal to $1$ was used in \cite{HenkinPolyakov1} to solve an interpolation problem in the polydisc, while a proof of the weighted formula for the more general class of analytic polyhedra appears in \cite{HenkinPolyakov2}. According to \cite{HenkinPolyakov2}, Henkin's formula gives uniform estimates for the $\bar\partial-$equation in the sup-norm. In addition, Henkin's formula also yields a bounded solution operator that preserves H\"older regularity for $\alpha \in (0,1)$. Stated precisely, we have the following theorem. 
\begin{theorem} \label{k = 0 main result}
Let $0< \alpha < 1$ and  $g \in Z^\alpha_{(0,1)}(\mathbb{D}^n)$ be a $\bar\partial-$closed H\"older$-\alpha$, $(0,1)-$form in the distributional sense. Then Henkin's solution operator to the $\bar\partial-$equation, restricts to a bounded linear operator $$H:  Z^\alpha_{(0,1)}(\mathbb{D}^n) \rightarrow C^\alpha(\mathbb{D}^n).$$
\end{theorem}

In fact, the solutions produced by Henkin's formula agrees with the solutions produced by Nijenhuis and Woolf in \cite{NijenhuisWoolf} and studied by Pan and Zhang in \cite{Pan2} and \cite{Zhang}.

\begin{definition}{(\textit{Nijenhuis and Woolf's Formula})} \newline
Let $k \geq 0, k\in \mathbb{Z}$. For any $f \in C^{k+\alpha}(D)$, $z \in D$, let $$T_j[f](z) = \frac{1}{2\pi i}\int_{\mathbb{D}}\frac{f(z_1,...,\zeta_j,z_{j+1},...,z_n)}{\zeta_j -z_j}d\zeta_j\wedge d\bar\zeta_j$$
$$S_j[f](z) = \frac{1}{2\pi i}\int_{\partial \mathbb{D}}\frac{f(z_1,...,\zeta_j,z_{j+1},...,z_n)}{\zeta_j -z_j}d\zeta_j.$$
Let 
$$\tilde{S}_{1} = id,\text{ } \tilde{S}_{k} = S_{k-1}...S_{1}, \text{ for all } 1 < k \leq n.$$
For any $g \in Z_{(0,1)}(\mathbb{D}^n)$, $g = \sum_i g_i d\bar{z_i}$ define Nijenhuis and Woolf's formula to be
$$T[g] = \sum_{j=1}^n T_{j}\tilde{S}_{j}[g_{j}],$$
then $\bar\partial T[g] = g$.
\end{definition}
\begin{remark}
As mentioned in the introduction, direct estimates of the operator norm of $T$ for $n \geq 2$ using Nijenhuis and Woolf's formula suggest that it loses (arbitrarily small) amounts of H\"older regularity. This is due to the fact that the Cauchy integral operators $S_j$ lose H\"older regularity in parameters, as observed in \cite{Tumanov} and \cite{Pan2}. However, both Henkin's formula and Nijenhuis and Woolf's formula in fact describe the \textit{same} solution operator. That is, $T = H$. This is the content of the subsequent lemma. Later, we will exploit this relationship to show that $H$ (and therefore $T$) in fact \textit{preserves} H\"older regularity.
\end{remark} 

\begin{lemma}\label{K and H are equal lemma}
$KT \equiv KH \equiv  0$. Consequently, for all $k \in \mathbb{Z}_{\geq0}$,$\alpha\in(0,1)$ and $f \in Z^{k+\alpha}_{(0,1)}(\mathbb{D}^n)$ we have $T[f] = H[f]$ 
\end{lemma}

\begin{proof}
Suppose for the moment that $KT \equiv KH \equiv  0$. Fix $f \in Z^{k+\alpha}_{(0,1)}(\mathbb{D}^n)$ and let $w := H[f] - T[f]$. Then $\bar\partial w= f - f = 0$ so that $w$ is holomorphic. Hence, $w = K[w] = KH[f] - KT[f] = 0$. So $T[f] = H[f]$.

It remains to show that $KH \equiv 0$ and $KT \equiv 0$. To see the former, fix $g \in Z_{(0,1)}(\bar{\mathbb{D}}^n)$, and let $d^2\zeta_j = d\bar{\zeta}_j\wedge d \zeta_j$. Note that for $|z_1| =... = |z_n| = 1$, $H$ is the sum of terms each with a region of integration given by  $$\gamma_J(z) = \{\zeta \in \mathbb{D}^n : \zeta_{j_1} = z_{j_1},...,  \zeta_{j_r} = z_{j_r}, 1 \geq |\zeta_{k_1}| = ... = |\zeta_{k_{n-r}}| \}.$$ for some $r$-tuple $J = (j_1, ... j_r)$. Hence, $K[H[g]]$ is given by $K$ applied to the sum of terms of the form 
\begin{align*}
\begin{aligned}
&\int_{|\zeta_{k_1}|=... = |\zeta_{k_{n-r}}|\leq 1} \frac{g_{k_s}(\zeta_{k_1},...\zeta_{k_{n-r}},z_J)}{(\zeta_{k_1}-z_{k_1})...(\zeta_{k_{n-r}}-z_{k_{n-r}})} d\zeta_{k_1}\wedge ...\wedge d^2\zeta_{k_s}\wedge ...\wedge d\zeta_{k_{n-r}} \\ 
=\text{ }&\int_{|\zeta_{k_s}|\leq 1}\bigg(\int_{|\zeta_{k_1}|=...=\widehat{|\zeta_{k_s}|}=...|\zeta_{k_{n-r}}|\leq 1} \frac{g_{k_s}(\zeta_{k_1},...\zeta_{k_{n-r}},z_J)}{(\zeta_{k_1}-z_{k_1})...\widehat{(\zeta_{k_s}- z_{k_s})}...(\zeta_{k_{n-r}}-z_{k_{n-r}})} d\zeta_{k_1} ...\widehat{d\zeta_{k_s}} ...d\zeta_{k_{n-r}}\bigg)  \frac{d^2\zeta_{k_s}}{(\zeta_{k_s}- z_{k_s})} \\
=\text{ }&T_{k_s}[G].
\end{aligned}
\end{align*}
Here, $$G(z_{1},...,\zeta_{k_s},....,z_{n}) = \int_{|\zeta_{k_1}|=...=\widehat{|\zeta_{k_s}|}=...|\zeta_{k_{n-r}}|\leq 1} \frac{g_{k_s}(\zeta_{k_1},...\zeta_{k_{n-r}},z_J)}{(\zeta_{k_1}-z_{k_1})...\widehat{(\zeta_{k_s}- z_{k_s})}...(\zeta_{k_{n-r}}-z_{k_{n-r}})} d\zeta_{k_1}...\widehat{d\zeta_{k_s}}...d\zeta_{k_{n-r}}.$$

However, since $S_jT_j = 0$ for all $j = 0,...n$ and $K = S_n...S_1$ by virtue of the fact that these operators commute, $KT_{k_s}[G] = 0$. Hence $KH[g] = 0$. Since $g$ was arbitrary, $KH = 0$. Likewise, since $T[g] = \sum_{j=1}^n T_{j}\tilde{S}_{j}[g_{j}]$, we have $KT = 0$.

\end{proof}
\begin{remark}
As a consequence of the proof of Lemma \ref{K and H are equal lemma}, there is a unique solution operator to the $\bar\partial$ equation with vanishing cauchy torus integral. In addition, we may view both $H$ and $T$ as different formulae for this canonical solution operator. As $T$ and $H$ describe the same solution operator, the results from Pan and Zhang in \cite{Pan2} and \cite{Pan1} obtained with the formula $T$ carry over to Henkin's formula, $H$. Hence, the following proposition holds. 
\end{remark}
\begin{proposition} \label{image of H is continuous proposition}
    For all $k \geq 0, k \in \mathbb{Z}$ and $0 < \alpha < 1$, Henkin's solution operator $H$ is a bounded linear operator $$H: Z^{k+\alpha}_{(0,1)}(\mathbb{D}^n) \rightarrow C(\bar{\mathbb{D}}^n).$$
    Moreover, it is the canonical solution operator such that $K H = 0$.
\end{proposition}
\begin{remark}
    This result can be obtained from Nijenhuis and Woolf's formula $T$ by repeated application of the one dimensional results that  $T_j[f](z) = \frac{1}{2\pi i}\int_{\mathbb{D}}\frac{f(z_1,...,\zeta_j,z_{j+1},...,z_n)}{\zeta_j -z_j}d\zeta_j\wedge d\bar\zeta_j$ preserves H\"older regularity, and $S_j[f](z) = \frac{1}{2\pi i}\int_{\partial \mathbb{D}}\frac{f(z_1,...,\zeta_j,z_{j+1},...,z_n)}{\zeta_j -z_j}d\zeta_j$ loses only infinitesimally small H\"older regularity. This argument is described carefully in \cite{Pan2}.
\end{remark}

\section{Proof of Theorem \ref{k = 0 main result}}
Fix $g  =  \sum_{i =1}^n g_i d\bar\zeta_j  \in Z^{k + \alpha}_{(0,1)}(\bar{\mathbb{D}}^n)$. In order to prove Theorem \ref{k = 0 main result} we make a short digression. Fix a permutation $\sigma$ of $\{1,...,n\}$ and let $D_\sigma$ denote the open sectors  $\{z \in \mathbb{D}^n : |z_{\sigma(1)}|>...>|z_{\sigma(n)}|\}$. Note that the closures of these sectors cover $\mathbb{D}^n$. Hence, in view of Proposition \ref{image of H is continuous proposition}, to obtain $C^\alpha$ estimates of $H[g]$, we need only obtain the same estimates on each of the open sectors $D_\sigma$. Moreover, by definition of the regions of integration $\gamma_J(z)$ in Theorem \ref{Henkins formula theorem}, it can be seen that for each $z \in D_\sigma$, the only terms in $H[g](z)$ which do not vanish take the following simple form (up to multiplication by a constant).

$$\int_{|\zeta_{k_1}|=... = |\zeta_{k_{n-r}}|\leq |z_{j_r}|} \frac{g_{k_s}(\zeta_{k_1},...\zeta_{k_{n-r}},z_{J_\sigma})}{(\zeta_{k_1}-z_{k_1})...(\zeta_{k_{n-r}}-z_{k_{n-r}})} d\zeta_{k_1}\wedge...\wedge d^2\zeta_{k_s}\wedge...\wedge d\zeta_{k_{n-r}}.$$ Here, $(j_1,...,j_r)$ is an ordered $r$-tuple in $\{1,...,n\}$ and $J_\sigma = (\sigma(j_1),....,\sigma(j_r))$. This motivates the following definition.
\begin{definition}
Let $a, b \in \mathbb{D}$, $\zeta, z  \in \mathbb{D}^q$,  $\zeta = (\zeta_1,...,\zeta_q)$, $z = (z_1,...,z_q)$ and $h\in C^\alpha(\mathbb{D}^q\times\mathbb{D}\times\mathbb{D})$. We write $d^2\zeta_j$ for $d\bar\zeta_j \wedge d\zeta_j$ and define 
 $$P[h](z,a,b) = \int_{|\zeta_1|= ... = |\zeta_q| \leq |a|} \frac{h(\zeta,a,b)}{(\zeta_1 - z_1)...(\zeta_q - z_q)} d^2\zeta_1 \wedge... \wedge d\zeta_q.$$
\end{definition}

The following lemmas apply.
\begin{lemma} \label{estimate 1}
For each $s \in \{0,...,q\}$, let $D_s = \{(z,a,b) \in \mathbb{D}^{q+2} : |z_1| < ... < |z_s| < |a| < |z_{s+1}|< ... < |z_q|\}$. Fix some $t \in  \{1,... q\}$. Then on each open sector $D_s$,  $P[h]$ is H\"older$-\alpha$ uniformly in the variable $z_t$, with coefficient independent of $z, a, b$, and proportional to $\|h\|_{C^\alpha(\mathbb{D}^{q+2})}$.
\end{lemma}
\begin{proof}
We first consider the case $t = 1$. Observe that
\begin{align*}
    \begin{aligned}
    P[h](z,a,b) &= \int_{|\zeta_1|=...=|\zeta_q| \leq |a|}\frac{h(\zeta,a,b)}{(\zeta_1-z_1)...(\zeta_q-z_q)}d^2\zeta_1 \wedge d\zeta_2\wedge...\wedge d\zeta_q \\
    &= \int_{|\zeta_1|\leq |a|}\bigg(\int_{|\zeta_2|=...=|\zeta_q| = |\zeta_1| }\frac{h(\zeta,a,b)}{(\zeta_2-z_2)...(\zeta_q-z_q)}d\zeta_2\wedge...\wedge d\zeta_q\bigg) \frac{d^2\zeta_1}{(\zeta_1- z_1)} \\
    &=:\int_{|\zeta_1|\leq |a|}\tilde h(\zeta_1,z_2,...,z_q,a,b) \frac{d^2\zeta_1}{(\zeta_1- z_1)}.
    \end{aligned}
\end{align*}
We claim that $\tilde h$ is uniformly bounded with coefficient independent of $z, a, b$, and proportional to $\|h\|_{C^\alpha(\mathbb{D}^{q+2})}$. Indeed,
\begin{align*}
    \begin{aligned}
        \tilde h(\zeta_1,z_2,...,z_q,a,b) &= \int_{|\zeta_2|=...=|\zeta_q| = |\zeta_1| }\frac{h(\zeta,a,b)}{(\zeta_2-z_2)...(\zeta_q-z_q)}d\zeta_2\wedge...\wedge d\zeta_q \\ 
        &= \int_{|\xi_2| = ... = |\xi_q|= 1}\frac{h(\zeta_1, |\zeta_1|\xi_2,...|\zeta_1|\xi_q,a,b)}{(\xi_2-\frac{z_2}{|\zeta_1|})...(\xi_q-\frac{z_q}{|\zeta_1|})}d\xi_2\wedge...\wedge d\xi_q.\\
    \end{aligned}
\end{align*}
Since the function $(\zeta_1,\xi_2,...,\xi_q,a,b) \mapsto h(\zeta_1, |\zeta_1|\xi_2,...|\zeta_1|\xi_q,a,b)$ is $C^\alpha$ on the region of integration, its Cauchy torus integral is uniformly bounded, so $\tilde h$ is uniformly bounded as well. Moreover, $\|\tilde h\|_{L^\infty} \lesssim \|h\|_{C^{\alpha}(\mathbb{D}^{p+q})}$. Furthermore, since $P[h]$ is obtained by applying the Cauchy integral on a domain to $\tilde h$, we see that for every $0<\epsilon <1$,  $P[h](z,a,b)$ is $C^{1-\epsilon}$ in $D_s$ uniformly in $z_1$. In particular, $P[h]$ is $C^\alpha$ on the sector $D_s$ uniformly in the variable $z_1$.
\newline 

We now consider the case $t \not= 1$. Observe that $\zeta_t\bar\zeta_t = \zeta_1\bar\zeta_1$, Therefore $d\bar\zeta_1 = \frac{\zeta_t}{\zeta_1}d\bar\zeta_t$. Hence, after ignoring changes in sign,
$$P[h](z,a,b) = \int_{|\zeta_1|=...=|\zeta_q| \leq |a|}\frac{\zeta_t\zeta_1^{-1}h(\zeta,a,b)}{(\zeta_1-z_1)...(\zeta_q-z_q)}d\zeta_1 \wedge...\wedge d^2\zeta_t \wedge...\wedge d\zeta_q.$$
The same analysis as before shows that $P[h](z,a,b)$ is $C^\alpha$ in $z_t$ uniformly with coefficient independent of $z, a, b$, and proportional to $\|h\|_{C^\alpha(\mathbb{D}^{q+2})}$.
\end{proof}

\clearpage
\begin{lemma} \label{estimate 2}
On each open sector $D_s$,  $P[h]$ is H\"older$-\alpha$ uniformly in the parameters $a$, $b$ with coefficient independent of $z, a, b$, and proportional to $\|h\|_{C^\alpha(\mathbb{D}^{q+2})}$.
\end{lemma}

\begin{proof}
We first show that $P[h]$ is $C^\alpha$ uniformly in $b$. Indeed, let $\epsilon \in \mathbb{C}$.
\begin{align*}
    \begin{aligned}
        |P[h](z,a,b+\epsilon) - P[h](z,a,b)| &\leq \int_{|\zeta_1| = ... =|\zeta_q| \leq |a|} \frac{\|h\|_{C^\alpha(\mathbb{D}^{q+2})}|\epsilon|^\alpha}{|\zeta_1 - z_1|... |\zeta_q-z_q|}|d^2\zeta_1|...|d\zeta_q| \\
        &\leq \|h\|_{C^\alpha(\mathbb{D}^{q+2})}|\epsilon|^\alpha \int_0^1\int_0^{2\pi}...\int_0^{2\pi} \frac{r^q \cdot d\theta_1 ... d\theta_q dr}{|re^{i\theta_1} - z_1|...|re^{i\theta_q} - z_q|} \\
        &\lesssim \|h\|_{C^\alpha(\mathbb{D}^{q+2})}|\epsilon|^\alpha \int_0^1\int_0^{2\pi}...\int_0^{2\pi} \frac{r^q \cdot d\theta_1 ... d\theta_q dr}{|r\theta_1 + |r -|z_1|||...|r\theta_q + |r -|z_q|||} \\ 
        &\lesssim \|h\|_{C^\alpha(\mathbb{D}^{q+2})}|\epsilon|^\alpha \int_0^1\big[\ln{(r\theta_1 + |r- |z_1||)}\big]_0^{2 \pi}...\big[\ln{(r\theta_q + |r- |z_q||)}\big]_0^{2\pi} dr \\
        &\lesssim \|h\|_{C^\alpha(\mathbb{D}^{q+2})}|\epsilon|^\alpha \int_0^1 |\ln{(|r- |z_1||)}|....|\ln{(|r- |z_q||)}|dr \\ 
        &\lesssim \|h\|_{C^\alpha(\mathbb{D}^{q+2})}|\epsilon|^\alpha \int_0^1 \sum_{i = 1}^q|\ln{(|r- |z_i||)}|^q dr \\
        &\lesssim \|h\|_{C^\alpha(\mathbb{D}^{q+2})}|\epsilon|^\alpha \int_0^1 |\ln{(r)}|^q dr \\
        &\lesssim \|h\|_{C^\alpha(\mathbb{D}^{q+2})}|\epsilon|^\alpha.
    \end{aligned}
\end{align*}
Similarly, we can show $P[h]$ is $C^\alpha$ uniformly in $a$.
\begin{multline*}
    P[h](z,a+\epsilon, b) - P[h](z,a,b) = \int_{|\zeta_1| = ... = |\zeta_q| \leq |a + \epsilon|}\frac{h(\zeta,a + \epsilon,b) - h(\zeta, a, b)}{(\zeta_1 -z_1)...(\zeta_q -z_q)}d^2\zeta_1\wedge...\wedge d\zeta_q \\
    + \int_{|a| \leq |\zeta_1| = ... = |\zeta_q| \leq |a + \epsilon|}\frac{h(\zeta,a,b)}{(\zeta_1 -z_1)...(\zeta_q -z_q)} d^2\zeta_1\wedge...\wedge d\zeta_q =: I_1 + I_2.
\end{multline*}
We estimate $I_1$ and $I_2$ separately.
For $I_1$, we may repeat the same analysis as we did for $b$ to see that $|I_1| \lesssim \|h\|_{C^\alpha(\mathbb{D}^{q+2})}|\epsilon|^\alpha$. Hence we need only estimate $I_2$.

\begin{align*}
    \begin{aligned}
        |I_2| &= \bigg|\int_{|a|}^{|a+ \epsilon|}\int_0^{2\pi}...\int_0^{2\pi}\frac{h(re^{i\theta_1},...,e^{ri\theta_q},a,b)}{(re^{i\theta_1} - z_1)...(re^{i\theta_q} - z_q)}r d\theta_1 (ire^{i\theta_2} d\theta_2)...(ire^{i\theta_q}d\theta_q)dr\bigg| \\ 
        &\lesssim |\epsilon| \bigg|\int_0^{2\pi}...\int_0^{2\pi}\frac{h(|a|e^{i\theta_1}...|a|e^{i\theta_q},a,b)}{(|a|e^{i\theta_1} - z_1)...(|a|e^{i\theta_q} - z_q)}|a|^q d\theta_1 (ie^{i\theta_2} d\theta_2)...(ie^{i\theta_q}d\theta_q)\bigg| \\
        &\lesssim |\epsilon|\bigg| \int_{|\xi_1| = 1}...\int_{|\xi| = 1}\frac{h(|a|\xi,a,b)\overline{i\xi_1}}{(\xi_1 - \frac{z_1}{|a|})...(\xi_q - \frac{|z_q|}{|a|})} d\xi_1 \wedge... \wedge d\xi_q \bigg|\\ 
        &\lesssim |\epsilon|\cdot \|h\|_{C^\alpha(\mathbb{D}^{q+2})}.
    \end{aligned}
\end{align*}
Here, the last inequality holds as the function $(\zeta,a,b) \mapsto h(|a|\xi,a,b)\overline{i\xi_1}$ is $C^\alpha$, so that its Cauchy torus integral is bounded, depending only on $\|h\|_{C^\alpha(\mathbb{D}^{q+2})}$. In particular, $P$ preserves H\"older$-\alpha$ regularity in the parameters $a,b$.
\end{proof}
\clearpage

By combining the results in Lemma \ref{estimate 1} and Lemma \ref{estimate 2}, we obtain the following proposition.
\begin{proposition}\label{combined estimates}
Let $0<\alpha <1$, then for each open sector $D_s$,  $P : C^\alpha(\mathbb{D}^q \times \mathbb{D} \times \mathbb{D}) \rightarrow C^\alpha(D_s)$ is a bounded linear operator.
\end{proposition}
\begin{proof}
By the preceding two lemmas, $P[h]$ is $C^{\alpha}$ in $D_s$ uniformly in each variable $z_t$ and parameters $a,b$, with coefficient proportional to $\|h\|_{C^\alpha(\mathbb{D}^{q+2})}$.
Hence $\|P[h]\|_{C^\alpha(D_s)} \lesssim \|h\|_{C^\alpha(\mathbb{D}^{q+2})}$.
\end{proof}

Theorem \ref{k = 0 main result} immediately follows from Proposition \ref{combined estimates}.
\begin{proof}{(\textit{Theorem \ref{k = 0 main result}})}
Observe that the closures of the finitely many open sectors $D_\sigma$  cover the polydisc $\mathbb{D}^n$. By Proposition 1, $H[g]$ is uniformly continuous on the polydisc $\mathbb{D}^n$. Therefore to show that $H:  Z^\alpha_{(0,1)}(\mathbb{D}^n) \rightarrow C^\alpha(\mathbb{D}^n)$ is a bounded linear operator, we need only show that given H\"older-$\alpha$ datum $g \in Z^\alpha_{(0,1)}(\mathbb{D}^n)$, the solution $H[g]$ is $C^\alpha$ on each of the open sectors $D_\sigma$, with coefficient proportional to $\|g\|_{C^\alpha({\mathbb{D}^n})}$. But this follows immediately from Proposition \ref{combined estimates}.
\end{proof}

\section{Proof of Theorem \ref{main theorem}}
Having shown Theorem \ref{k = 0 main result}, we see that for $0< \alpha < 1$, $H  :  Z^\alpha_{(0,1)}(\mathbb{D}^n) \rightarrow C^\alpha(\mathbb{D}^n)$ is a bounded linear operator. To complete the proof of Theorem \ref{main theorem}, we once again exploit the fact that $T = H$ and induct on $k$ using  Nijenhuis and Woolf's formula $T:  Z^{k +\alpha}_{(0,1)}(\mathbb{D}^n) \rightarrow C^{k+\alpha}(\mathbb{D}^n)$ with Theorem \ref{k = 0 main result} as the base case.
\begin{proof}{(\textit{Theorem 1})}
Fix an integer $k \geq 0$. As noted in \cite{Vekua}, we may differentiate the singular integral $T_j$ with respect to $z_j$ and $\bar z_j$. From which we obtain $$\frac{\partial}{\partial z_j}T_j[g](z) = \Pi_j[g](z)  := \frac{1}{2\pi i}\int_{\mathbb{D}}\frac{g(z_1,...,\zeta_j,z_{j+1},...,z_n)}{(\zeta_j -z_j)^2}d\zeta_j\wedge d\bar\zeta_j$$ and $$\frac{\partial}{\partial \bar{z_j}}T_j[g](z) = g.$$
Suppose now $T = H :  Z^{k +\alpha}_{(0,1)}(\mathbb{D}^n) \rightarrow C^{k+\alpha}(\mathbb{D}^n)$ is a bounded linear operator. Fix $g \in  Z^{k +1 +\alpha}_{(0,1)}(\mathbb{D}^n) $.

Then, as $T$ is a solution operator to the $\bar\partial-$equation, we have that  $\bar\partial T[g] = g$. In addition, for all $j = 1,..., n$,  $\frac{\partial}{\partial z_j}T_j[g](z) = \Pi_j[g](z) $ and $\frac{\partial}{\partial z_j} S_j[g](z) = S_j[\frac{\partial g}{\partial z_j}](z)$. Therefore by commuting the differentiation symbol with the operator when possible, we obtain the formula $$\frac{\partial}{\partial z_j}T[g](z) =  \sum_{i \not=  j}T_i\tilde{S_i}[\frac{\partial}{\partial \zeta_j}g_i] + \Pi_j\tilde{S}_j[g_j].$$

Since $k+1 \geq 1$, we may apply Stokes' Theorem, to the term $\Pi_j\tilde{S}_j[g_j]$ as in \cite{Vekua} to obtain $$\Pi_j\tilde{S}_j[g_j] = T_j\tilde S_j[\frac{\partial}{\partial \zeta_j}g_j] - \frac{1}{2\pi i} \int_{\partial \mathbb{D}}\frac{\tilde S_j[g_j](z_1,...,z_{j-1},\zeta_j,...z_n)}{\zeta_j - z_j} d\bar\zeta_j  = T_j\tilde S_j[\frac{\partial}{\partial \zeta_j}g_j] + S_j[\tilde S_j [g_j]\cdot \bar\zeta_j^2].$$

From which it immediately follows that  $$\frac{\partial}{\partial z_j}T[g](z) =  \sum_{i}T_i\tilde{S_i}[\frac{\partial}{\partial \zeta_j}g_i] +  S_j[\tilde S_j [g_j]\cdot \bar\zeta_j^2] = T[\frac{\partial}{\partial \zeta_j}g] + S_j[\tilde S_j [g_j]\cdot \bar\zeta_j^2].$$
By the induction hypothesis, $\|T[\frac{\partial}{\partial \zeta_j}g]\|_{C^{k+\alpha}(\mathbb{D}^n)} \lesssim \|\frac{\partial}{\partial \zeta_j}g\|_{C^{k+\alpha}(\mathbb{D})}\lesssim\|g\|_{C^{k+1+\alpha}(\mathbb{D})}$. Furthermore, for all $j$, $g_j \in C^{k + 1 + \alpha}(\mathbb{D}^n)$ we have $\|S_j[\tilde S_j [g_j]\cdot \bar\zeta_j^2]\|_{C^{k+\alpha}} \lesssim \|g_j\|_{C^{k+1+\alpha}}$ as the Cauchy integral loses arbitrarily small amounts of H\"older regularity in parameters. This completes the induction.
 \end{proof}

\section{Conclusion and extensions}
In this note, we show the existence of a solution operator $H$ on product domains that preserves H\"older regularity. It is canonical in the sense that $KH = 0$ where $K$ is the Cauchy integral over the torus. The proof rests on a careful analysis of both Henkin's formula as well as Nijenhuis and Woolf's formula for solving $\bar\partial u = g$ on the polydisc. Both formulae hold in the more general case of product domains. Indeed, a version of Henkin's formula holds for analytic polyhedra which includes product domains, while Nijenhuis and Woolf's formula holds with no changes. Therefore in principle, the same proof strategy can be employed to extend the result to product domains. However, without explicit formulae for the boundary, the estimates in Lemmas \ref{estimate 1} and \ref{estimate 2} become more technical. Hence for simplicity, we only discuss the case of the polydisc.

\printbibliography[
heading=bibintoc,
title={References}]
\end{document}